\documentclass{amsart}
\usepackage{amsthm,amsfonts,amsmath,amscd,amssymb,latexsym,epsfig,texdraw}
\usepackage[title,titletoc]{appendix}
\newcommand{\gl}{{\mathfrak g \mathfrak l}}

\newcommand{\cx}{{\mathbb C}}

\newcommand{\tr}{\operatorname{tr}}

\newcommand{\Lie}{\operatorname{Lie}}
\newcommand{\Hom}{\operatorname{Hom}}

\newcommand{\supp}{\operatorname{supp}}

\newcommand{\Mat}{\operatorname{Mat}}

\newcommand{\Rep}{\operatorname{Rep}}

\newcommand{\wh}{\widehat}
\newcommand{\ol}{\overline}
\newcommand{\ul}{\underline}

\numberwithin{equation}{section}

\newtheorem{theorem}{Theorem}[section]

\newtheorem{lemma}[theorem]{Lemma}

\newtheorem{corollary}[theorem]{Corollary}

\newtheorem{proposition}[theorem]{Proposition}

\theoremstyle{remark}

\newtheorem{remark}[theorem]{Remark}

\newtheorem{definition}[theorem]{Definition}

\newtheorem{example}[theorem]{Example}

\newcommand{\oP}{{\mathbb{P}}}

\newcommand{\oZ}{{\mathbb{Z}}}

\newcommand{\sF}{{\mathcal{F}}}

\newcommand{\sO}{{\mathcal{O}}}

\newcommand{\sT}{{\mathcal{T}}}

\newcommand{\sV}{{\mathcal{V}}}

\begin{document}

\title{Quivers and Poisson structures}
\author{Roger Bielawski }
\address{School of Mathematics\\
University of Leeds\\Leeds LS2 9JT\\ UK}



\begin{abstract} We produce natural quadratic Poisson structures on moduli spaces of representations of quivers. In particular, we study a natural Poisson structure for the generalised Kronecker quiver with $3$ arrows.
 \end{abstract}

\maketitle

\thispagestyle{empty}

\section{Introduction}

This note is partly an original research article and partly an exposition and interpretation, from the  viewpoint of a   ``commutative geometer",  of  constructions and some of the results of Van den Bergh \cite{vB} and  Pichereau and Van de Weyer \cite{PW} on Poisson structures on moduli spaces of representations of quivers. 
\par
Our starting point is the observation that, given a quiver $Q$, it is easy to produce natural tangent bivectors on its representation spaces. Recall that a representation of $Q$ in $\ul V=(V_1,\dots,V_k)$ consists of a family of matrices $X_a$, for each arrow $a$. The vector space of all such representations is denoted by $\Rep(Q,\ul V)$. A cotangent vector to $\Rep(Q,\ul V)$ can be identified with a collection of transposed matrices $C_a$, $a\in Q$. A natural quadratic bivector field $\Pi\in \Lambda^2 T\Rep(Q,\ul V)$ can be defined by
\begin{equation} \Pi\bigl((C_a),(D_a)\bigr)|_{(X_a)}=\sum R_{abcd}\tr\left( X_aC_bX_cD_d-X_aD_bX_cC_d\right),\label{one}\end{equation}
where $R_{abcd}$ are scalars and the sum is over all arrows $a,b,c,d$ such that $ab^\ast cd^\ast$ is a closed path (here the asterisk denotes the reverse arrow). This bivector field is clearly invariant under $\prod_{i=1}^k GL(V_i)$ and, so, it descends to each GIT quotient of $\Rep(Q,\ul V)$.
\par
The condition which the $R_{abcd}$ must satisfy, in order for $\Pi$ to be a Poisson structure, turns out to be the {\em associative Yang-Baxter equation} \cite{Aguiar} 
$$ r_{12}r_{23}+r_{31}r_{12}+r_{23}r_{31}=0,$$
for an appropriate associative algebra depending on $Q$ (see \S\ref{quad}).
\par
While this statement can be proved directly, it is more illuminating to approach the problem  via noncommutative geometry.
There are currently two approaches to noncommutative Poisson structures for any algebra $A$. One, due to Crawley-Boevey \cite{CB} produces a Lie algebra structure on $A/[A,A]$, and it induces Poisson structures on moduli spaces of semisimple representations. The other approach, based on double brackets and double derivations, is due to Van den Bergh \cite{vB} and it induces invariant Poisson structures on any representation space. In the case of quivers, the graded Lie algebra of double derivations has been reinterpreted by Pichereau and Van de Weyer \cite{PW} in terms of the graded necklace Lie algebra, introduced earlier by Lazaroiu \cite{L}.
\par
Part of the purpose of this work is to explain how natural and inevitable it is, even if one is interested only in commutative differential geometry,  to  arrive at the graded necklace picture (see \S\ref{poly}). In fact, once we have this setup, the proof that \eqref{one} defines a Poisson structure on every $\Rep(Q,\ul V)$ if and only if the $R_{abcd}$ satisfy the associative Yang-Baxter equation is almost trivial - see \S\ref{quad}. 
\par
Throughout the paper, we give many examples of noncommutative and of the corresponding commutative Poisson brackets. The commutative geometry proves to be a useful guide to constructing noncommutative examples. In particular, in \S\ref{contract} we describe two procedures, which turn a Poisson structure on one quiver into a Poisson structure on another quiver.
\par
We are particularly interested in generalised Kronecker quivers, given their close relation to moduli spaces of sheaves \cite{ACK}. In section \ref{leaves}, we study a quadratic Poisson structure for  the generalised Kronecker quiver with $3$ arrows, obtained from a solution of the associative Yang-Baxter equation, found by Aguiar \cite{Aguiar}. We show that the induced Poisson structure on the moduli spaces of representations is generically nondegenerate and we describe the corresponding symplectic form. This allows to find generic symplectic leaves of this Poisson structure.

\section{Quivers and moduli spaces of their representations\label{quivers}}

This section serves to recall basic facts about quivers and to set the notation.
\par
A quiver is an oriented graph $Q$. As usual, we also use $Q$ to denote the set of arrows in a quiver. We write $I=\{1,\dots,k\}$ for the set of vertices, and  $h,t:Q\to I$ for the maps which associate to each arrow its head and tail.
\par
A representation of $Q$ is a collection $\ul W$ of vector spaces $W_i$, $i\in I$, and linear maps $\phi_a:W_{t(a)}\to W_{h(a)}$, $a\in Q$.  The vector space of representations of $Q$ in $\ul W$
is  $\bigoplus_{a\in Q} \Hom(W_{t(a)},W_{h(a)})$. We fix isomorphisms $W_i\simeq \cx^{\alpha_i}$, $i=1,\dots,k$, and speak simply of representations with  dimension vector  $\alpha=(\alpha_1,\dots,\alpha_k)$. The space of such representations is denoted by  $\Rep(Q,\alpha)$. Thus
$$\Rep(Q,\alpha)=\bigoplus_{a\in Q} \Mat_{\alpha_{h(a)}\times \alpha_{ t(a)}}(\cx),
$$
and it is acted upon by the group
\begin{equation} G=\prod_{i=1}^k GL_{\alpha_i}(\cx).\label{G}\end{equation} 

\subsection{The path algebra}
The {\em path algebra} $\cx Q$ of $Q$ is the associative algebra over $\cx$ generated by arrows (including trivial arrows at each vertex) with multiplication given by concatenation of paths (written from right to left). Let $e_i$, $i=1,\dots,k$, denote the trivial path at vertex $i$ and $B=\bigoplus_{i=1}^k \cx e_i$ be the corresponding commutative semisimple algebra. Then $\cx Q$ is a $B$-algebra, i.e. a $\cx$-algebra $A$ equipped with a  $\cx$-algebra homomorphism $\imath:B\to A$.

Representations of a quiver $Q$ can be interpreted  as representations of the path algebra $\cx Q$. 
Let $\alpha=(\alpha_1,\dots,\alpha_k)$ be a dimension vector, and let $|\alpha|=\sum_{i=1}^k\alpha_i$. We view a matrix $Z\in \Mat_{|\alpha|\times |\alpha|}(\cx)$ as a block matrix $Z_{ij}$, $i,j=1,\dots,k$, with $Z_{ij}$ an $\alpha_i\times \alpha_j$-matrix. We consider $B=\bigoplus_{i=1}^k \cx e_i$ as diagonally embedded in $\Mat_{|\alpha|\times |\alpha|}(\cx)$ with $e_i$ the $\alpha_i\times \alpha_i$ identity matrix. This makes $\Mat_{|\alpha|\times |\alpha|}(\cx)$ into a $B$-algebra equipped with a canonical $B$-bimodule trace map
\begin{equation}
 \tr_B:\Mat_{|\alpha|\times |\alpha|}(\cx)\to B,\quad \tr_B Z=\bigl(\tr Z_{11},\dots,\tr Z_{kk}\bigr).\label{trace}
\end{equation}
Then:
$$ \Rep(Q,\alpha)\simeq \Hom_B\bigl(\cx Q, \Mat_{|\alpha|\times |\alpha|}(\cx)\bigr).$$
More precisely, let $X\in \Rep(Q,\alpha)$ be given by a collection of matrices $X_a\in \Mat_{\alpha_{h(a)}\times \alpha_{ t(a)}}(\cx)$, $a\in Q$. We view each $X_a$ as an $|\alpha|\times |\alpha|$ matrix $\tilde{X}_a$, by making the $\alpha_{h(a)}\times \alpha_{ t(a)}$ block of  $\tilde{X}_a$ equal to $X_a$ and remaining blocks $(\tilde{X}_a)_{ij}$ equal to zero. Then the homomorphism $\phi(X)\in \Hom_B\bigl(\cx Q, \Mat_{|\alpha|\times |\alpha|}(\cx)\bigr)$ evaluates a $P\in \cx Q$ on matrices $\tilde X_a$. In other words, if $x_1\dots x_n$, $x_i\in Q$, is a path, then $\phi(X)(P)$ is:
\begin{equation} P(X)=\tilde X_{x_1}\cdots \tilde X_{x_n}.\label{evaluate}
\end{equation}

\subsection{Moduli spaces} The interesting varieties arise by taking appropriate quotients of $\Rep(Q,\alpha)$ by the group $G$. There is the ordinary algebro-geometric quotient:
$$\Rep(Q,\alpha)/\hspace{-1mm}/ G={\rm Spec}\bigl(\cx[\Rep(Q,\alpha)]^G\bigr),$$
the points of which are closed orbits in $\Rep(Q,\alpha)$. More generally, one takes a $G$-invariant open subset of $\Rep(Q,\alpha)$, consisting of points which are semistable in an appropriate sense, and considers orbits which are closed in this subset. In particular, for any  character $\chi:G\to \cx^\ast$, there is the ring $\cx[\Rep(Q,\alpha)]^{G,\chi^n}$ of relative invariants of weight $\chi^n$, consisting of functions  $f\in \cx[\Rep(Q,\alpha)]$ such that $f(gx)=\chi^n(g) f(x)$.
The corresponding GIT-quotient is then \cite{King}:
$$\Rep(Q,\alpha)/\hspace{-1mm}/_\chi G={\rm Proj}\Bigl(\bigoplus_{n\geq 0}\cx[\Rep(Q,\alpha)]^{G,\chi^n}\Bigr).$$
The corresponding semistable subset $\Rep(Q,\alpha)^{\chi\text{-ss}}$ consists of points $x$, for which there exists an $f\in  \cx[\Rep(Q,\alpha)]^{G,\chi^n}$ with $n\geq 1$ and $f(x)\neq 0$.  The variety $\Rep(Q,\alpha)/\hspace{-1mm}/_{\chi} G$ is projective over $\Rep(Q,\alpha)/\hspace{-1mm}/ G$. Its smooth locus corresponds to $\chi$-stable orbits, i.e. those of maximal dimension, closed in $\Rep(Q,\alpha)^{\chi\text{-ss}}$.

\section{Equivariant and invariant polyvector fields\label{poly}}

Let $Q$ be a quiver. The double  $\ol Q$ of $Q$ is obtained by attaching, for each arrow $a$ in $Q$, an arrow $a^\ast$ with the same endpoints, but opposite direction.
We shall view elements of the path algebra $\cx \ol Q$ of the double quiver as {\em noncommutative equivariant contravariant tensor fields}. Let us first explain the analogous commutative concept.
\par
Let $M$ be a complex  manifold and let $R$ be a holomorphic ring  (i.e. $R$ is also a manifold such that the ring operations are morphisms). An {\em $R$-valued contravariant tensor field} (of degree $r$) is a holomorphic map $X:\bigotimes^r T^\ast M\to R$, linear on fibres. We denote by $\sT^r(M;R)$ the vector space of such tensor fields and by $\sT(M;R)$ the direct sum of all $\sT^r(M;R)$. It becomes a graded associative algebra with respect to the tensor product and multiplication on $R$. 
\par
Suppose also that a complex Lie group $G$ acts on $M$, and via ring automorphisms on $R$. We can then consider $G$-equivariant $R$-valued tensor fields, which form a subalgebra $\sT^r(M;R)^G$ of $\sT^r(M;R)$. If, in addition, 
 $R$ is equipped with a $G$-invariant linear map $\tr:R\to \cx$, then for any $X\in \sT^r(M;R)^G$ we obtain a $G$-invariant tensor $\tr X$. 
By restricting such a tensor to $\Lambda^\ast T^\ast M$ or $S^\ast T^\ast M$, we obtain invariant skew-symmetric or symmetric tensors.
\par
Let us also recall that skew-symmetric contravariant tensors on $M$ are called {\em polyvector fields}, and their space (exterior algebra) is denoted by $\sV(M)$, with $\sV^r(M)$ denoting those of degree $r$. Finally, remark that the above construction is valid in other categories of manifolds, e.g. smooth or algebraic.

\medskip

Let now $\Rep(Q,\alpha)$ be a representation space of a quiver $Q$. We view $\Rep(Q,\alpha)$ as an algebraic manifold $M$, and, when we speak of functions, tensor or vector fields on $M$, we always mean {\em polynomial} functions, tensor or vector fields.
\par
We set $R=\Mat_{|\alpha|\times |\alpha|}(\cx)$ (recall that $|\alpha|=\sum_{i=1}^k\alpha_i$). To any element $P$ of $\cx\ol Q$, i.e. a noncommutative polynomial in arrows $a$ and $a^\ast$, $a\in Q$, we shall associate an element $\Phi(P)$ of  $\sT(M;R)^G$, where $G$ is defined in \eqref{G}. A point $m\in M$ is a collection of matrices $X_a\in \Mat_{\alpha_{h(a)}\times \alpha_{ t(a)}}(\cx)$, $a\in Q$. A covector is identified with a collection of transposed matrices, i.e. $T^\ast M\simeq \bigoplus_{a\in Q}\Mat_{\alpha_{t(a)}\times \alpha_{h(a)}}(\cx)$.
\par
Let now $P=\lambda x_1\dots x_n$ be a monomial in $\cx\ol Q$ with $r$ occurrences of dual arrows $a^\ast$, i.e. 
$$ P=\lambda P_1a_1^\ast\cdots P_r a_r^\ast P_{r+1},\quad P_i\in \cx Q.$$
Then we  define $\Phi(P)\in \sT^r(M;R)^G$ at a point $X\in M$ by setting $\Phi(P)(Y^1,\dots,Y^r)$, $Y^i=(Y^i_a)\in T^\ast M$,  to be the matrix
$$ \lambda P_1(X)\tilde Y_{a_1}^1\cdots  P_r(X)\tilde Y_{a_r}^r P_{r+1}(X),$$
where $P_i(X)$ is given by \eqref{evaluate}, and $\tilde Y_{a_i}^i$ is the  $|\alpha|\times |\alpha|$-matrix defined just before that formula.
\par
We can grade $\cx\ol Q$ by the number of  arrows in $\ol Q\backslash Q$, i.e.
\begin{equation} \cx \ol Q=\bigoplus_{r=0}^{\infty}\cx\ol Q^r,
\end{equation}
where $\cx\ol Q^r$ is the linear subspace of $\cx \ol Q$ generated by monomials $x_1\dots x_n$ such that exactly $r$ of the $x_i$ belong to $\ol Q\backslash Q$. The map 
$$ \Phi:\cx\ol Q\to \sT(M;R)^G$$
becomes a homomorphism of graded $B$-algebras. 
\par
The trace map $\tr_B:R\to B$, defined in \eqref{trace} is $G$-invariant, and so we obtain a map
$$ \Psi_B=\tr_B\circ \Phi:\cx \ol Q\to \sT(M;B)^G.$$
Composing this with a further linear map $\mu:B\to \cx$ results in map  taking values in the algebra  $\sT(M)^G$ of $G$-invariant contravariant tensors. If $\mu(e_i)=\lambda_i$, we shall write $\tr_\lambda=\mu\circ \tr_B$, so that $\tr_\lambda:R\to \cx$ is $G$-invariant map given by
\begin{equation} \tr_\lambda(Z)=\sum_{i=1}^k \lambda_i \tr Z_{ii},\label{trace1}\end{equation}
where $Z_{ii}$ is the diagonal $\alpha_i\times \alpha_i$-block of $Z\in R$.

\subsection{The graded necklace Lie algebra}

We have now a family of natural maps $\tr_\lambda\circ \Phi:\cx\ol Q\to \sT(M)^G$ ($M=\Rep(Q,\alpha)$), which evaluates a trace of a noncommutative polynomial in $\cx\ol Q$ on $(Y_a)\in \bigotimes T^\ast_X M$. For each $X\in M$, we can restrict $\tr_\lambda \Phi(X)$ to the space of  skew-symmetric covariant tensors. This results in a map
\begin{equation}\Psi=\Psi_\lambda: \cx\ol Q\to \sV(M)^G,\label{Psi}\end{equation}
to $G$-invariant polyvector fields. Since trace of a product is invariant under cyclic permutations of factors, the value $\Psi(P)(Y_1,\dots,Y_r)$ for $P\in \cx\ol Q^r$ remains unchanged under a cyclic permutation of monomials in $P$ and a simultaneous induced permutation of $(Y_1,\dots,Y_r)$. On the other hand, the $Y_i$ are meant to anticommute. We are led to the following definition:
\begin{definition}
 Let $Q$ be a quiver. We define  $\sV Q$ to be the quotient of $\cx\ol Q$ by the relations
\begin{equation}
 PQ=(-1)^{pq}QP, \quad \text{if $P\in \cx\ol Q^p$, $Q\in \cx\ol Q^q$}.\label{pq}
\end{equation}
\end{definition}
In other words, the dual arrows $a^\ast$ ``cyclically anticommute''. Observe that paths which are not closed become zero in $\sV Q$, and so $\sV Q$ should be viewed as being generated by closed paths (``necklaces").
\par
We note that the map $\Psi$ descends to $\sV Q$, and that $\sV Q$ has the induced grading.
In particular, elements of $\sV^1 Q$ can be uniquely represented in the form
\begin{equation} \sum_{a\in Q} P_a a^\ast,\quad P_a\in \cx Q,\enskip e_{h(a)}P_a=P_a.\label{V1}\end{equation}
We can identify $\sV^1 Q$ with ${\rm Der}_B\;\cx Q$: the derivation corresponding to \eqref{V1} is defined on generators as $a\mapsto P_a$. In particular, $\sV^1 Q$ is canonically a Lie algebra, and the map $\Psi:\sV^1 Q\to (\sV^1 M)^G$ preserves the Lie bracket. 
\begin{example} For $\tr_\lambda:R\to \cx$ given by \eqref{trace1}, the map $\Psi_\lambda$ sends $Pa^\ast\in \sV^1 Q$ to the vector field $V\in \sV^1(M)^G$ given by
$$ V|_X=c\sum_{i,j} P(X)_{ij}\frac{\partial}{\partial a_{ij}},$$
where the $a_{ij}$ are coordinates on $M$ given by the entries of the matrix $X_a$, and $c=c(\lambda,a,P)$ depends linearly on $\lambda_i$, with $i$ in the set of all heads and tails of arrows making up $Pa^\ast$.
\end{example}

\par
We shall now extend the Lie bracket on $\sV^1 Q$ to all of $\sV Q$. 
\par
Let first define, for any $w\in \ol Q$, the ``directional superderivative'' $D_w:\sV Q\to \cx\ol Q$. For any monomial $x_1\dots x_n$, $x_i\in \cx\ol Q$
we set:
\begin{equation}
 D_w (x_1\dots x_n)=\sum_{x_i=w}(-1)^{\lambda_i\mu_i} x_{i+1}\dots x_nx_1\dots x_{i-1}, \label{D}
\end{equation}
modulo relations \eqref{pq}, where $\lambda_i$ (resp. $\mu_i$) is the number of dual arrows $a^\ast$ among $x_{i+1},\dots,x_n$ (resp. among $x_1,\dots,x_i$).  We extend $D_w$ linearly to all of $\sV Q$. 
\par
We now define the {\em Schouten bracket} on $\sV Q$ as follows. For $\gamma\in \sV^r Q$ and $\delta\in \sV^s Q$ we set (cf. \cite{L, PW}):
\begin{equation}
[\gamma,\delta]_{_{\sV}}=\sum_{a\in Q}\left(D_{a^\ast}(\gamma)D_a(\delta)-(-1)^{(r-1)(s-1)} D_{a^\ast}(\delta)D_a(\gamma)\right),
 \label{bracket}
\end{equation}
modulo relations \eqref{pq}.

\begin{proposition} Let $A,B,C\in \cx\ol Q$ represent elements of $\sV^r Q, \sV^s Q$ and $\sV^t Q$, respectively. The Schouten bracket on $\sV Q$
 satisfies:
\begin{itemize}
\item[(i)] $[A,B]_{_{\sV}}=-(-1)^{(r-1)(s-1)}[B,A]_{_{\sV}}$;
\item[(ii)] the graded Jacobi identity: 
$$(-1)^{(r-1)(t-1)}[A,[B,C]_{_{\sV}}]_{_{\sV}}+(-1)^{(s-1)(r-1)}[B,[C,A]_{_{\sV}}]_{_{\sV}}+(-1)^{(t-1)(s-1)}[C,[A,B]_{_{\sV}}]_{_{\sV}}=0.$$
\end{itemize}
\end{proposition}
\begin{proof}
 Part (i) is obvious from the definition. Part (ii) is a  straightforward, if lengthy, computation. Alternatively, it follows from the corresponding property for double derivations in \cite{vB} and Theorem 1 in \cite{PW}.
\end{proof}

 Therefore $\bigl(\sV Q,[\,,\,]_{_{\sV}}\bigr)$  is a graded Lie (super)algebra, called the {\em graded necklace Lie algebra of the quiver $Q$}.  As far as we know, it was first considered by Lazaroiu \cite{L}. It the present context, it has been studied by Pichereau and  Van de Weyer \cite{PW}.
 Its significance lies in the fact that the Schouten bracket on $\sV Q$ induces the standard Schouten bracket on $G$-invariant vector fields on any representation space $\Rep(Q;\alpha)$ of the quiver $Q$ ($G$ defined in \eqref{G}). Let us write $M=\Rep(Q;\alpha)$ and $\sV(M)^G$ for the $G$-invariant polynomial polyvector fields on $M$. Recall the formula for the Schouten bracket on polyvector fields on any manifold $M$:
$$ [X_1\cdots X_r,Y_1\cdots Y_s]=\sum_{i,j} (-1)^{i+j}[X_i,Y_j]X_1\cdots X_{i-1}X_{i+1}\cdots  X_rY_1\cdots Y_{j-1}Y_{j+1}\cdots Y_s,
$$
\vspace{-3mm}

where $X_i$ and $Y_j$ are vector fields and the product is the wedge product.
The following is a direct proof, without passing to double derivations, of a result known from \cite{L,PW}:
\begin{theorem}[van den Bergh-Pichereau-van de Weyer] Let $M=\Rep(Q;\alpha)$ be a representation space of a quiver $Q$. The map $\Psi:\sV Q\to \sV(M)^G$ commutes with the respective Schouten brackets.
\label{commute}\end{theorem}
\begin{proof} The manifold $M$ has canonical global coordinates $a^{ij}$ given by the entries of matrices $X_a$, $a\in Q$. In general, having chosen global coordinates $x_i$, $i=1,\dots, N$, we can write a polynomial vector field as
$X=\sum_{i=1}^N f_i(x)\frac{\partial}{\partial x_i}$, where $f_i(x)$ is polynomial in the $x_i$. For another such vector field $Y= \sum_{i=1}^N g_i(x)\frac{\partial}{\partial x_i}$, we put
$$ X(Y)= \sum_{i=1}^N (X.g_i)(x)\frac{\partial}{\partial x_i},$$
so that $[X,Y]=X(Y)-Y(X)$. We can then rewrite the above formula for the Schouten bracket of two polyvector fields as
\begin{multline} \sum_{i,j}  (-1)^{i(r-i)+(j-1)(s-j+1)}S_i X_i(Y_j) T_j \\- (-1)^{(r-1)(s-1)} (-1)^{j(s-j)+(i-1)(r-i+1)}T_j Y_j(X_i) S_i,\label{ST}\end{multline}
where $S_i$ and $T_j$ denote the following polyvector fields:
$$ S_i=X_{i+1}\cdots X_rX_1\dots X_{i-1},\quad\enskip T_j=Y_{j+1}\cdots Y_sY_1\dots Y_{j-1},$$
and the product is the wedge product. Observe that the exponents are the same as in \eqref{bracket} and \eqref{D}.
\par
If  $\gamma\in \sV^r Q$ and $ \delta\in \sV^s Q$ are represented, respectively, by monomials
$$ P_1a_1^\ast\cdots P_r a_r^\ast,\quad\enskip Q_1 b_1^\ast\cdots Q_s b_s^\ast,$$ then $\Psi(\gamma)(A)$, in the canonical global coordinates $a^{mn}$, is   equal to (up to a constant factor depending on the choice of $\tr_\lambda$):
$$ \sum P_1(A)_{m_1n_1}\frac{\partial}{\partial a_1^{m_2n_1}}\wedge P_2(A)_{m_2n_2}\frac{\partial}{\partial a_2^{m_3n_2}}\wedge \dots \wedge P_r(A)_{m_rn_r}\frac{\partial}{\partial a_r^{m_1n_r}},$$
where the sum ranges over all $m_1,\dots,m_r$ and $n_1,\dots,n_r$ describing the coordinates of matrices $P_i(A)$, $i=1,\dots,r$. A similar formula holds of course  for $\Psi(\delta)$. Now a computation shows that the Schouten bracket of $\Psi(\gamma)$ and $\Psi(\delta)$, written in the form \eqref{ST}, is the same as  $\Psi\bigl([\gamma,\delta]_{_{\sV}}\bigr)$.
\end{proof}

\section{Poisson structures\label{bloody}}

Once the Schouten bracket has been defined, Poisson structures are defined in the usual manner:
\begin{definition} Let $Q$ be a quiver. A {\em Poisson structure on $\cx Q$} is an element $\Pi$ of $\sV^2 Q$ such that $[\Pi,\Pi]_{_{\sV}}=0$.\label{fish}
\end{definition}
\begin{remark} This definition is due to Pichereau and Van de Weyer \cite{PW}, who reinterpreted Van den Bergh's double Poisson structures \cite{vB} in terms of the graded necklace Lie algebra (cf. \S\ref{relate} below).
\end{remark}

Theorem \ref{commute} implies immediately:
\begin{corollary}
 Let $\Pi$ be a Poisson structure on $\cx Q$ and let $\alpha$ be a dimension vector. Then $\Psi(\Pi)$ is a $G$-invariant Poisson structure on $\Rep(Q;\alpha)$.\hfill $\Box$
\end{corollary}

An element $\Pi$ of $\sV^2 Q$ can be written as
\begin{equation}
 \Pi=\sum_{a,b\in Q}\sum_{P,R} [P_{}a^\ast, R b^\ast],\label{Pi}
\end{equation}
where $P,R\in \cx Q$,  and the nonzero monomials in $Pa^\ast Rb^\ast$  correspond to closed paths. If $\Pi$ is a Poisson structure, then the Poisson bracket on coordinate functions in  $\Rep(Q;\alpha)$ is:
\begin{equation} \{a_{ij},b_{kl}\}(X)=  2c \sum_{P,R}P(X)_{kj}R(X)_{il},\label{induce}\end{equation}
where $c$ depends linearly on the $\lambda_i$ in \eqref{trace} used to determine 
 $\Psi=\tr_\lambda \Phi$.

\medskip

\begin{remark} By construction, the induced Poisson structure on each $\Rep(Q,\alpha)$ is $G$-invariant, and, hence, it descends to a Poisson structure on  $\Rep(Q,\alpha)/\hspace{-1mm}/ G$ and  $\Rep(Q,\alpha)/\hspace{-1mm}/_{\chi} G$, for any character $\chi$. As the quotient spaces are not necessarily smooth, a Poisson structure means here a Poisson bracket on the structure sheaf.\end{remark}

The above corollary has a converse, which shows that the above definition is indeed the most natural one.
\begin{theorem} Let $Q$ be a quiver and $\Pi\in \sV^2 Q$. If $\Psi(\Pi)$ is a Poisson structure on each $\Rep(Q;\alpha)$, then $\Pi$ is a Poisson structure on $\cx Q$, i.e. $[\Pi,\Pi]_{_{\sV}}=0$.\label{converse}
\end{theorem}
\begin{proof} Let $\gamma=[\Pi,\Pi]_{_{\sV}}$. We need to show that, if $\Psi(\gamma)=0$ on each 
$\Rep(Q;\alpha)$, then $\gamma=0$. The argument is the same as in \cite[pp.390--391]{Gin}. One considers the direct limit $V_\infty$ of $B$-modules over all finite-dimensional $B$-modules, and the ind-group $G_\infty$ obtained as the direct limit of the $G$ in \eqref{G}. The direct limit of the trace maps $\tr_\lambda$ is a map $\tr_\infty: \cx \ol Q/[\cx \ol Q, \cx \ol Q]\to \Rep(\cx\ol Q,V_\infty)^{G_\infty}$.  Since $\sV^3 Q$ is $\cx \ol Q^3$ modulo commutators, the statement follows from the following fact, which is proved as in the paper of Ginzburg \cite{Gin}, cited above: if $f\in \cx \ol Q$ satisfies $\tr_\infty(f)=0$, then $f\in [\cx \ol Q, \cx \ol Q]$.
\end{proof}

\subsection{Relation with Poisson structures of Van den Bergh and Crawley-Boevey\label{relate}}
 In \cite{vB}, Van den Bergh defines {\em double Poisson structures} on associative algebras via double brackets, i.e. bilinear maps $\{\hspace{-1mm}\{\,,\,\}\hspace{-1mm}\}:A\times A\to A\otimes A$, which are derivation in the second argument and are skew-symmetric, in the sense that 
$\{\hspace{-1mm}\{a,b\}\hspace{-1mm}\}=-\{\hspace{-1mm}\{b,a\}\hspace{-1mm}\}^\circ$, where $(u\otimes v)^\circ= v\otimes u$. It is clear that, for $A=\cx Q$, there is a natural correspondence between elements of $\sV^2 Q$ and double brackets over $B$: for $\Pi$ as in \eqref{Pi}, we define $\{\hspace{-1mm}\{\,,\,\}\hspace{-1mm}\}$ on generators $a\in Q$ by $\{\hspace{-1mm}\{a,b\}\hspace{-1mm}\}=P_{ab}\otimes Q_{ab}\in \cx Q\otimes \cx Q$. The condition $[\Pi,\Pi]_{_{\sV}}=0$ is equivalent to Van den Bergh's version of Jacobi identity. Thus Poisson structures on $\cx Q$ are in bijection with double Poisson structures, and such a corresponding pair $\Pi$, $\{\hspace{-1mm}\{\,,\,\}\hspace{-1mm}\}$ induces the same Poisson structure \eqref{induce} on any $\Rep(Q;\alpha)$. Van den Bergh's double Poisson structures for $\cx Q$ have been translated into the language of necklace Lie superalgebra by Pichereau and Van de Weyer \cite{PW}.

\medskip

On the other hand,  Crawley-Boevey defines in \cite{CB} {\em $H_0$-Poisson structures}  as Lie brackets $\{\,,\,\}$ on $A/[A,A]$ such that each $\{f,\cdot\}$ is induced by a derivation $d_f$ of $A$. It is known \cite[Lemma 2.6.2]{vB} that double Poisson structures induce $H_0$-Poisson structures, and, consequently, a Poisson structure on $\cx Q$ in the sense of Definition \ref{fish} induces an  $H_0$-Poisson structure on $\cx Q$.  Let us describe this explicitly.
\par
For any  $w\in Q$, define the ``directional derivative'' $D_w:\cx Q/[\cx Q,\cx Q]\to \cx Q$ by the formula \cite{Gin} (cf. \eqref{D})
\begin{equation} D_w (x_1\dots x_n)=\sum_{x_i=w} x_{i+1}\dots x_nx_1\dots x_{i-1}\enskip \mod [\cx Q,\cx Q],\label{Dw}\end{equation}
for any path $x_1\dots x_n$. Let $\Pi$ be a Poisson structure given by \eqref{Pi}. Then the corresponding $H_0$-Poisson structure is defined by:
$$  \{f,g\}=\sum_{a,b\in Q}\sum_{P,R} P D_a( f) R D_b(g) \enskip \mod [\cx Q,\cx Q],$$
for any $f,g\in \cx Q/[\cx Q,\cx Q]$. The derivation $d_f$ which induces $\{f,\cdot\}$ is determined by its values on generators as follows:
$$ b\mapsto \sum_{a\in Q}\sum_{P,R} PD_a( f) R,\quad b\in Q.$$

\subsection{Constant and linear Poisson structures}

\begin{definition} An element  $\Pi$ of the form \eqref{Pi} in $ \sV^2 Q$ is called {\em homogeneous of degree $p$} if the length of each nonzero closed path $P_{ab}a^\ast Q_{ab}b^\ast$ is $p+2$.\end{definition}

Such $\Pi$ induce homogeneous bivectors on each $\Rep(Q,\alpha)$ via \eqref{induce}. We have the obvious:
\begin{proposition} Let $\Pi\in \sV^2 Q$ be a Poisson structure on $\cx Q$ such that $\Pi=\Pi^1+\Pi^2$, where $\Pi^1$ and $\Pi^2$ are  homogeneous of different degrees. Then $\Pi^1$ and $\Pi^2$ are compatible Poisson structures on $\cx Q$, i.e. $ [\Pi_i,\Pi_j]_{_{\sV}}=0$ for $i,j=1,2$.\hfill $\Box$\label{homog}\end{proposition}

We discuss briefly constant and  linear bivectors in $\sV^2 Q$.
\par
 Clearly, any constant $\Pi$ defines a Poisson structure. Such a $\Pi$ is given by $|Q|\times |Q|$-matrix $J$, which is skew-symmetric and $J_{ab}=0$ if $a$ and $b$ do not form a cycle (i.e. $J_{ab}$ can be nonzero only if  $h(a)=t(b)$ and $t(a)=h(b)$). In particular, if $Q$ itself is a double quiver, $Q=\ol K$, then $J$ defined by $J_{aa^\ast}=1$, $J_{a^\ast a}=-1$, if $a\in K$, and all other entries of $J$ equal to zero, is a nondegenerate Poisson (symplectic) structure on $\cx \ol K$ considered by Ginzburg \cite{Gin} and many other authors. It induces the canonical symplectic structure on $\Rep(Q,\alpha)=T^\ast \Rep(K,\alpha)$. 
\par
Let us now discuss linear Poisson structures. We can represent a linear element of $\sV^2 Q$ uniquely as 
\begin{equation} \Pi=\sum_{x,a,b\in Q} J^{x}_{ab} xa^\ast b^\ast,\label{linear}\end{equation}
where $J^x_{ab}$ must be zero unless $xa^\ast b^\ast$ is a closed path. There is no skew-symmetry condition. We compute the Schouten bracket $[\Pi,\Pi]_{_{\sV}}$ from \eqref{bracket}:
$$ [\Pi,\Pi]_{_{\sV}}=4\sum_{x,a,b,c\in Q}\sum_{y\in Q}\left(J^x_{ay}J^y_{bc}-J^x_{yc}J^y_{ab}\right)xa^\ast b^\ast c^\ast.$$
The vanishing of $[\Pi,\Pi]_{_{\sV}}$ is, therefore, equivalent to $J^x_{ab}$ being structure constants of an associative algebra structure on $\cx^{|Q|}$. If we write the standard basis of $\cx^{|Q|}$ as $\{e_a;\, a\in Q\}$, then the multiplication is given by 
$$e_a\cdot e_b= \sum_x J^x_{ab}e_x,\quad a,b,x\in Q.$$
The algebra structure is clearly compatible with the canonical $B$-bimodule structure on $\cx^{|Q|}$, and so
\begin{proposition}[Pichereau and Van de Weyer \cite{PW}] Linear Poisson structures on $\cx Q$ are in 1-1 correspondence with associative algebra structures on $\cx^{|Q|}$, compatible with the canonical $B$-bimodule structure.\hfill $\Box$\label{ass1}
\end{proposition}

More explicitly, let  us decompose $\cx^{|Q|}$ as $\bigoplus_{i,j=1}^k E_{i,j}$, where $E_{i,j}$ is generated by the $e_a$ with $t(a)=i$ and $h(a)=j$ (i.e. by arrows from $i$ to $j$), and $E_{i,j}=0$ if there are no arrows from $i$ to $j$. Then a reformulation of Proposition \ref{ass1} is:
\begin{proposition} Linear Poisson structures on $\cx Q$ are in 1-1 correspondence with associative algebra structures on $\cx^{|Q|}$, such that
\begin{itemize}
\item[(i)] $E_{i,j}E_{k,l}=0$ if $j\neq k$;
\item[(ii)] $E_{i,j}E_{j,k}\subset E_{i,k}$.\hfill $\Box$
\end{itemize}\label{ass}
\end{proposition}
In particular, there are no nontrivial linear Poisson structures on $\cx Q$, if there are no arrows $a$ such that there is a path of length $2$ from $t(a)$ to $h(a)$.
\begin{example}[cf. \cite{PW}, Proposition 10] Let $Q$ consist of one vertex with $n$ loops attached to it. The linear Poisson structures on $\cx Q$ (the free algebra on $n$ letters) correspond to $n$-dimensional associative algebras. Each such algebra induces a $GL_m(\cx)$-invariant Poisson structure on $\Rep(Q,m)\simeq \gl_m(\cx)\otimes \cx^n$. Since these Poisson structures are linear, they are actually Lie-Poisson structures on $\gl_m(\cx)\otimes \cx^n$. In particular, the Poisson structure induced from the algebra $\bigoplus_{i=1}^n \cx$ (direct sum of $1$-dimensional algebras)  is the direct sum of $n$ copies of the standard Lie-Poisson structure on $\gl_m(\cx)^\ast$.
\par
On the other hand, for $n=2$, the noncommutative Poisson brackets $xx^\ast x^\ast+ yx^\ast y^\ast$ and $xx^\ast x^\ast+ yx^\ast y^\ast + yy^\ast x^\ast $ correspond to semidirect product Lie algebras $\gl_m(\cx)\ltimes \Mat_{m\times m}(\cx)$, where $\gl_m(\cx)$ acts  on $\Mat_{m\times m}(\cx)$ by left matrix multiplication in the first case, and by conjugation in the second case.
 \label{free}
\end{example}
\begin{example} Let $Q$ be the following quiver:
\begin{equation*}
\begin{texdraw}
\drawdim mm
\arrowheadsize l:1 w:1
\arrowheadtype t:H

\lcir r:2
\rmove (-2 0)
\ravec (0 -0.5)
\rmove (4 0.5)
\fcir f:0 r:0.7
\rmove (5.2 0)
\lellip rx:5 ry:1
\rmove (0 1)
\ravec (-1 0)
\rmove (1 -2)
\ravec (1 0)
\rmove (4.4 1)
\fcir f:0 r:0.7

\rmove (2 0)
\lcir r:2
\rmove (2 0)
\ravec (0 -0.5)
\rmove (4 0.5)

\end{texdraw}
\label{Q2}
\end{equation*}
Let us label the loops $a$ and $b$, and the two other paths $x$ and $y$. There is a natural affine (sum of linear and constant) Poisson structure on $\cx Q$ given by:
\begin{equation} aa^\ast a^\ast + bb^\ast b^\ast+ x^\ast y^\ast.\label{aff}\end{equation}
If $\alpha=(k,l)$ is a dimension vector, then $\Rep(Q,\alpha)$ is isomorphic to
$$ \Mat_{k\times k}(\cx)\oplus \Mat_{l\times l}(\cx)\oplus \Mat_{k\times l}(\cx)\oplus \Mat_{l\times k}(\cx).
$$
The induced Poisson structure on $\Rep(Q,\alpha)$ is a linear combination  of the standard Lie-Poisson structures on $\gl_{k}(\cx)^\ast$ and $\gl_l(\cx)^\ast$ and the symplectic structure $\tr dX\wedge dY$ on $\Mat_{k\times l}(\cx)\oplus \Mat_{l\times k}(\cx)$. The quotient Poisson structure on $\Rep(Q,\alpha)/\hspace{-1mm}/ G$, where $G=GL_k(\cx)\times GL_l(\cx)$, corresponds to rank $r$ deformations \cite{AHP}, with symplectic leaves given by coadjoint orbits of loop groups and many interesting applications. 
\par
On the other hand, $\cx Q$ admits also a linear Poisson structure, obtained from  the associative algebra structure of $\Mat_{2\times 2}(\cx)$ on  $\cx^{|Q|}$. The resulting Poisson structure on $\Rep(Q,\alpha)$ is a Lie-Poisson structure on $\Mat_{(k+l)\times (k+l)}(\cx)$,  depending on the choice of $\tr_\lambda$ used to  define $\Psi_\lambda$ in \eqref{Psi}.
\label{butterfly}\end{example}

\section{Quadratic Poisson structures\label{quad}} 

As observed in the previous section, most quivers do not have nontrivial linear Poisson structures, simply because usually there are no cycles of length $3$ in $\cx \ol Q$. On the other hand, there are always cycles of length $4$ and, so, we  turn to quadratic Poisson structures. We shall only consider elements of $\sV^2 Q$ of the form
\begin{equation} \Pi=\sum_{x,y,a,b\in Q} R^{xa^\ast}_{yb^\ast} xa^\ast y b^\ast,\label{quadratic}
\end{equation} 
where $R^{xa^\ast}_{yb^\ast}\in \cx$ and $R^{xa^\ast}_{yb^\ast}=-R^{yb^\ast}_{xa^\ast}$. In other words,  we do not allow monomials of the form $xya^\ast b^\ast$. We shall call such $\Pi$ {\em uniform}.
\par
We compute the superderivatives $D_w(\rho)$ of $\rho=xa^\ast y b^\ast$ using \eqref{D}:
$$ D_x(\rho)=a^\ast y b^\ast,\enskip D_y(\rho)=-b^\ast x a^\ast, \enskip D_{a^\ast}(\rho)=-yb^\ast x,\enskip D_{b^\ast}(\rho)=x a^\ast y,$$
if $x\neq y$ and $a\neq b$, or sums of corresponding terms, if that is not true. All other  $D_w(\rho)$ are zero. Therefore:
$$ [\Pi,\Pi]_{_{\sV}}=4\sum \sum_{z\in Q} \left(R^{pu^\ast}_{qz^\ast}R^{zv^\ast}_{rw^\ast} + R^{rw^\ast}_{pz^\ast}R^{zu^\ast}_{qv^\ast} +   R^{qv^\ast}_{rz^\ast}R^{zw^\ast}_{pu^\ast}\right)pu^\ast qv^\ast r w^\ast,$$
where the first sum is over orbits of the cyclic group $\oZ_3$ on all closed paths $pu^\ast qv^\ast r w^\ast$. Thus, \eqref{quadratic} defines a Poisson structure on $\cx Q$ if and only if the constants $R^{xa^\ast}_{yb^\ast}$ satisfy the equation
\begin{equation} \sum_{z\in Q} R^{pu^\ast}_{qz^\ast}R^{zv^\ast}_{rw^\ast} + R^{rw^\ast}_{pz^\ast}R^{zu^\ast}_{qv^\ast}   +  R^{qv^\ast}_{rz^\ast}R^{zw^\ast}_{pu^\ast}=0 \label{YB}
\end{equation}
for all $p,q,r,u,v,w\in Q$ forming a closed paths. This is an example of the {\em associative Yang-Baxter equation} \cite{Aguiar}. To explain this concept,  consider an element $r$ of $A\otimes A$, where $A$ is an arbitrary associative algebra with $1$. Define $r_{12}$ to be the element $r\otimes 1$ of $A\otimes A\otimes A$, and let $r_{23},r_{31}$ be obtained from $r_{12}$ by cyclic permutation of tensor factors (i.e. $r_{23}=1\otimes r$, etc.) Then $r$ satisfies the associative Yang-Baxter equation (to be abbreviated as AYB) if
\begin{equation} r_{12}r_{23}+r_{31}r_{12}+r_{23}r_{31}=0.\label{YB2}
\end{equation}
Observe that the left-hand side makes sense for non-unital algebras, and, therefore, we can also speak of solutions to AYB for algebras without $1$.
\par
In the case of a quiver $Q$, the  relevant associative algebra
$A_Q$ is generated (over $\cx$) by all (nonzero) paths of length $2$ of the form $xa^\ast$, $x,a\in Q$, with the product
$$ (xa^\ast)\cdot (yb^\ast)= \delta_{ay}\, xb^\ast,$$
$\delta_{ay}$ being the usual Kronecker symbol. Once again, $A_Q$ is a $B$-bimodule, and 
a $\Pi$ of the form \eqref{quadratic} can be viewed as an element $r$ of $A_Q\otimes_{_B} A_Q$:
$$ r=\sum_{x,y,a,b\in Q}R^{xa^\ast}_{yb^\ast} xa^\ast \otimes y b^\ast.$$
It is immediate that \eqref{YB} is equivalent to $r$ being a solution of AYB in $A_Q$. Hence:
\begin{proposition} Uniform quadratic Poisson structures on $\cx Q$ are in 1-1 correspondence with skew-symmetric solutions $r\in A_Q\otimes_{_B} A_Q$ of the associative Yang-Baxter equation in the algebra $A_Q$. \hfill $\Box$
\end{proposition}

\begin{remark} Observe that $A_Q$ with the commutator is isomorphic to  the Lie algebra ${\rm Der}_1\cx Q$ of linear derivations of $\cx Q$. A skew-symmetric solution to AYB is, in particular, a solution to the classical Yang-Baxter equation:
$$ [r_{12},r_{23}]+[r_{31},r_{12}]+[r_{23},r_{31}]=0,$$
and, therefore, any uniform quadratic Poisson structure on $\cx Q$ arises from a classical triangular $r$-matrix in ${\rm Der}_1\cx Q\otimes_{_B}{\rm Der}_1\cx Q$.
 The condition \eqref{YB2} is, however, stronger than the classical Yang-Baxter equation.
\end{remark} 
\begin{example} Let $Q$ be a generalised Kronecker quiver, consisting of $n$ arrows from $1$ to $2$. Then $A_Q\simeq \Mat_{n\times n}(\cx)$, and $A_Q\otimes_{_B} A_Q\simeq \Mat_{n\times n}(\cx)\otimes_{_\cx} \Mat_{n\times n}(\cx) $.
\par
More generally, if $Q$ is a quiver with $k$ vertices and $n$ arrows such that all arrows end at $1$, then $A_Q\simeq \Mat_{n\times n}(\cx)$, and 
$$A_Q\otimes_{_B} A_Q\simeq \bigoplus_{i,j=1}^k\Mat_{n_i\times n_j}(\cx)\otimes_{_\cx} \Mat_{n_j\times n_i}(\cx), $$ 
where $n_i$ is the number of arrows ending at $i$. Similarly, if all arrows of $Q$ start at $1$, then this last formula holds with $n_i$ denoting the number of arrows ending at $i$.
\end{example}
\begin{example} The previous example shows desirability of having solutions to AYB in $\Mat_{m\times m}(\cx)$. An example of such a solution is given by Aguiar \cite[Example 2.3.3]{Aguiar}:
\begin{equation} r=\sum_{i,j=1}^{m-1}\sum_{k=1}^{\max(i,j)} e_{i,i+j-k+1}\wedge e_{j,k},\label{agu}\end{equation}
where $e_{i,j}$ denote the elementary matrices.\label{rmat}
\end{example}
\begin{example} For $m=2$, the last example gives $r=e_{12}\wedge e_{11}$. For any quiver $Q$ with two arrows $x,y$ ending in vertex $1$, this gives a skewsymmetric element  $xy^\ast\otimes xx^\ast - xx^\ast\otimes xy^\ast $ of $A_Q\otimes A_Q$.  This element belongs to $A\otimes_{_B} A_Q$, if $Q$ is the Kronecker quiver (two arrows from $2$ to $1$), or if $Q$ has one vertex and two loops (i.e. for $\cx Q=\cx\langle x, y\rangle$). In the latter case, the resulting noncommutative Poisson structure $\Pi=xy^\ast xx^\ast - xx^\ast xy^\ast $ has been discussed in detail in \cite[\S 6.3]{PW}.\label{gl2}
\end{example}
\begin{example}
For $m=3$, $r$ and the corresponding $\Pi$ are: 
$$r=e_{12}\wedge e_{11}+e_{13}\wedge e_{21} + e_{23}\wedge e_{11}+e_{23}\wedge e_{22},$$
$$\Pi=[xy^\ast, xx^\ast]+ [xz^\ast,yx^\ast]+[yz^\ast,xx^\ast] +[yz^\ast,yy^\ast].$$
Again, if $Q$ is the generalised Kronecker quiver (with $3$ arrows), or has $3$ loops attached to one vertex, then $\Pi$ is a Poisson structure. In particular, in both cases $\Rep(Q,1)\simeq \cx ^3$, and the Poisson structure $\Psi(\Pi)$ on $\cx^3$ is:
$$ x^2\frac{\partial}{\partial y}\wedge \frac{\partial}{\partial x}+2xy \frac{\partial}{\partial z}\wedge \frac{\partial}{\partial x} +y^2 \frac{\partial}{\partial z}\wedge \frac{\partial}{\partial y}.
$$
This Poisson structure is invariant under diagonal action of $\cx^\ast$ and descends to $\cx P^2=\Rep(Q,1)/\hspace{-1mm}/ \cx^\ast$ (here $Q$ is the Kronecker quiver with $3$ arrows). The corresponding section of $\Lambda^2 T\cx \oP^2 \simeq \sO(3)$ is $-z(x^2)+y(2xy)-x(y^2)$, i.e. $xy^2-x^2 z$.
\label{gl3}
\end{example}

\begin{example} Already the simplest example, discussed in Example \ref{gl2} above, produces an interesting Poisson structure
on representation spaces. Let $\Pi=[xy^\ast, xx^\ast]$ on $\cx Q$, where $Q$ is the Kronecker quiver (by fusing the two vertices, we also obtain a Poisson structure on $\cx\langle x, y\rangle$).  The induced Poisson structure $\pi=\Psi(\Pi)$ on 
$\Rep(Q, (m,n))=\Mat_{n\times m}(\cx)\oplus \Mat_{n\times m}(\cx)$ is
$$ \pi_{(X,Y)}\bigl((A_1,B_1),(A_2,B_2))= \tr\bigl( XB_1 XA_2- XB_2XA_1\bigr),$$
where $(A_i,B_i)$, $i=1,2$, are cotangent vectors.
\par
In particular, if $n=m$ and $X$ is invertible, then this is a generically nondegenerate Poisson structure. The corresponding symplectic form 
is
$$ \omega= \tr d(X^{-1})\wedge dY$$
(the calculation proceeds as in the proof of Proposition \ref{four}).
Therefore $\pi$ is an extension of the Poisson structure induced by this  symplectic form on $GL_n(\cx)\times \gl_n(\cx)$ to $\Mat_{n\times n}(\cx)\oplus \Mat_{n\times n}(\cx)$. \label{neww}
\end{example} 

\section{Contractions of quivers\label{contract}}

We shall now describe  certain operation on quivers, which allows to view an open subset of a moduli space of representations of one quiver as the 
(full) moduli space of another quiver. It is  analogous to considering standard open affine subsets of $\cx P^n$.
\par
\begin{definition} Let $Q$ be a quiver and $a\in Q$ an arrow between two distinct vertices $i$ and $j$. We define a new quiver $Q_a$ by fusing vertices $i$ and $j$, and removing the arrow $a$. We shall refer to $Q_a$ as {\em $Q$ with $a$ contracted}.\label{Qa}
\end{definition}
 Let  $\alpha$ be a dimension vector such that $\alpha_i=\alpha_j$. If $X=(X_b)\in  \Rep(Q,\alpha)$ is such that $X_a$ is invertible, then we can use the action of $GL_{\alpha_j}$ to make $X_a=1$. It follows that
\begin{equation}\left\{X=(X_b)\in  \Rep(Q,\alpha);\: \det X_a\neq 0\right\}/GL_{\alpha_j}(\cx)\simeq  \Rep(Q_a,\alpha^\prime),\label{red}\end{equation}
where $\alpha^\prime=(\alpha_1,\dots, \wh{\alpha_j},\dots,\alpha_k)$. 
\par
In addition, let  $\chi:G\to \cx^\ast$ a character such that $\chi|_{GL_{\alpha_i}(\cx)}=\chi|_{GL_{\alpha_j}(\cx)}$, and define the character $\chi^\prime$ on $G^\prime=\prod_{s\neq j} GL_{\alpha_s}$ as the restriction of $\chi$ on $\prod_{s\neq i,j} GL_{\alpha_s}$, and $\chi^\prime=1$ on $ GL_{\alpha_i}$. Then $\chi$-semistable representations of $Q$ with invertible $X_a$ correspond to $\chi^\prime$-semistable representations of $Q_a$, and 
the corresponding GIT quotients are isomorphic.
\par
Suppose now that we are given a Poisson structure $\Pi$ on $\cx Q$. It induces a $G$-invariant Poisson structure $\pi$ on $\Rep(Q,\alpha)$ and so, owing to \eqref{red}, a $G^\prime$-invariant Poisson structure $\pi_a$ on $\Rep(Q_a,\alpha^\prime)$. We aim to describe this induced Poisson structure directly in terms of the quiver $Q_a$.\\ Let $U_a=\left\{X=(X_a)\in  \Rep(Q,\alpha);\: X_a\in GL_{\alpha_j}(\cx)\right\}$ and $H=GL_{\alpha_j}$. Since $H$ acts freely and properly on $U_a$, the Poisson structure on $U_a/H$ is obtained by restricting the Poisson structure $\pi$ on $U_a$ to $X_a=1$ and to covectors which annihilate the tangent vectors generated by the action of $H$. Let $S_j=\{b\in Q; h(b)=j\}$ and $T_j=\{b\in Q; t(b)=j\}$. The vector field $\check{\rho}$, generated by $\rho\in \Lie(H)$, equals, at $X=(X_b)$, to 
$$\sum_{b\in S_j} \rho X_b- \sum_{b\in T_j}  X_b \rho,
$$
and it follows that  a covector $Y=(Y_b)$ annihilates all such vectors if and only if 
\begin{equation}  \sum_{b\in S_j} X_bY_b- \sum_{b\in T_j}  Y_b X_b =0.\label{moment}\end{equation}
The arrow $a$ belongs to $S_j$, and we conclude that at points of $\Rep(Q,\alpha)$, where $X_a=1$, we need to evaluate $\pi$ on covectors which satisfy
\begin{equation*} Y_a=\sum_{b\in T_j}  Y_b X_b\:- \sum_{b\in S_j,\,b\neq a} X_bY_b.
\end{equation*}
Therefore $\pi_a=\Psi(\Pi_a)$, where $\Pi_a\in \sV^2 Q_a$ is obtained from $\Pi\in \sV^2 Q$ by replacing $a$ with $1$ and $a^\ast$ with
\begin{equation*} \sum_{b\in T_j} b^\ast b\:- \sum_{b\in S_j,\,b\neq a} bb^\ast.
\end{equation*}
\begin{proposition} Let $Q$ be a quiver, $a\in Q$ with $h(a)\neq t(a)$, and let $Q_a$ be $Q$ with $a$ contracted. If $\Pi$ is a Poisson structure on $\cx Q$, then $\Pi_a$ is Poisson structure on $\cx Q_a$. \label{lll}
\end{proposition}
\begin{proof} For any dimension vector $\beta$, $\Rep(Q_a,\beta)$ arises as $U_a/GL_{\alpha_j}$ for some $\Rep(Q,\alpha)$. The above argument shows that
$\Psi(\Pi_a)$ is an invariant Poisson structure on each $\Rep(Q_a,\beta)$. The statement follows now from Theorem \ref{converse}.
\end{proof}
\begin{remark} One has a completely analogous operation, using $GL_{\alpha_i}(\cx)$ instead of $GL_{\alpha_j}(\cx)$. Again, we obtain a Poisson structure $\Pi^\prime_a$ on $\cx Q_a$ from  a $\Pi$, by setting $a=1$ and 
\begin{equation*} a^\ast=  \sum_{b\in S_i,\,b\neq a} bb^\ast \: -\sum_{b\in T_i} b^\ast b.
\end{equation*}
\end{remark}

\begin{example} Consider again the Poisson bracket $\Pi=[xy^\ast, xx^\ast]$ on the Kronecker quiver (cf. Examples \ref{gl2} and \ref{neww}). Contracting the arrow $x$ gives the bracket $\Pi_x=yy^\ast y^\ast$ on $Q_x$, which has one vertex and one loop. The corresponding Poisson structure on $\Rep(Q_a,k)$ is the standard Lie-Poisson structure on $\Mat_{k\times k}(\cx)$.
On the other hand, contracting $y$ results in the cubic bracket $[xx^\ast, x^2x^\ast]$ on $Q_y=Q_x$. The induced Poisson structure on $\Mat_{k\times k}(\cx)$ is the extension of the symplectic form $\tr X^{-1} d(X^{-1})\wedge d(X^{-1})$ on invertible matrices.\end{example}

\begin{example} Let $Q$ be the Kronecker quiver with $3$ arrows and let $\Pi$ be the quadratic Poisson structure described in Example \ref{gl3}. We contract $Q$ at any of the arrows to obtain a quiver  consisting  of one vertex and two loops.
The corresponding Poisson structures on the free algebra on two letters are:
\begin{equation*}
\Pi_x=- [y^\ast, yy^\ast+zz^\ast]-[z^\ast, y^2y^\ast +yzz^\ast]-[yz^\ast, zz^\ast].
\end{equation*}
\begin{equation*}
\Pi_y=-[x^2x^\ast+xzz^\ast, xx^\ast]+[z^\ast, x^\ast x-zz^\ast].
\end{equation*}
\begin{equation*}
\Pi_z=[xy^\ast,xx^\ast]-[x^2x^\ast+xyy^\ast, yx^\ast]-[yxx^\ast+ y^2y^\ast,xx^\ast + yy^\ast].
\end{equation*}
\label{exe}
\end{example}

Using Proposition \ref{homog}, we obtain:
\begin{corollary} The free algebra $\cx\langle x,y\rangle$ admits the following pairs of compatible Poisson brackets:
$$ [xx^\ast, x^\ast]+[yy^\ast, x^\ast] \quad\text{\rm and}\quad [yy^\ast,xy^\ast]-[y^\ast,x^2x^\ast+xy y^\ast],$$
$$ [xx^\ast, x^\ast]+[yx^\ast, y^\ast] \quad\text{\rm  and}\quad [yy^\ast, yxx^\ast+y^2y^\ast],$$
$$ [xy^\ast,xx^\ast] \quad\text{\rm  and}\quad [x^2x^\ast+xyy^\ast, yx^\ast]+[yxx^\ast+ y^2y^\ast,xx^\ast + yy^\ast].$$
\label{reduced}\end{corollary} 

\begin{remark} The compatibility of these pairs follows also from the fact that the ones of the higher degree are infinitesimal deformations of the ones with the lower degree: see Section \ref{leaves}.\end{remark}

\subsection{A generalisation} We now generalise the procedure described above to the case of several arrows meeting at a vertex. 
\begin{definition} Let $Q$ be a quiver and $a_1,\dots,a_s\in Q$ arrows such that $h(a_m)=j$, $m=1,\dots,s$, $t(a_m)=v_m$ with $v_m\neq j$ and $v_m\neq v_{m^\prime}$, $m,m^\prime=1,\dots,s$. We define a new quiver $Q^\prime$ with vertices $1,\dots, \wh{j},\dots,k$ and arrows defined as follows:
\begin{itemize} 
\item[(i)] if $a\in Q$ and $h(a)\neq j$, $t(a)\neq j$, then $a$ is unchanged in $Q^\prime$;
\item[(ii)] if $a$ is one of the $a_m$, then $a$ is removed;
\item[(iii)] if $t(a)=j$ and $h(a)\neq j$, then $a$ is replaced by $s$ arrows, starting at $v_m$, $m=1,\dots,s$, and ending in $h(a)$; 
\item[(iv)] if $h(a)=j$ and $t(a)\neq j$, then $a$ is replaced by $s$ arrows, starting at $t(a)$ and ending in $v_m$, $m=1,\dots,s$; 
\item[(v)] if $t(a)=h(a)=j$, then $a$ is replaced by a loop at each $v_m$ and two arrows in opposite directions between each pair of distinct $v_m, v_{m^\prime}$, $m,m^\prime=1,\dots,s$.
\end{itemize}\label{long}
\end{definition}
Essentially, every arrow beginning or ending in $j$ is decomposed into arrows beginning or ending in the $v_m$, $m=1,\dots,s$.
\begin{example} Let $Q$ be a quiver with $3$ vertices,  two arrows from $1$ to $3$, and two arrows from $2$ to $3$. If we take $a_1$ to be an arrow from $1$ to $3$ and $a_2$ an arrow from $2$ to $3$, then the above procedure produces the quiver 
\begin{equation*}
\begin{texdraw}
\drawdim mm
\arrowheadsize l:1 w:1
\arrowheadtype t:H

\lcir r:2
\rmove (-2 0)
\ravec (0 -0.5)
\rmove (4 0.5)
\fcir f:0 r:0.7
\rmove (5.2 0)
\lellip rx:5 ry:1
\rmove (0 1)
\ravec (-1 0)
\rmove (1 -2)
\ravec (1 0)
\rmove (4.4 1)
\fcir f:0 r:0.7

\rmove (2 0)
\lcir r:2
\rmove (2 0)
\ravec (0 -0.5)
\rmove (4 0.5)

\end{texdraw}.
\end{equation*}
This quiver and its significance have been already mentioned in Example \ref{butterfly}.
\end{example}

\medskip

Let now $\alpha$ be a dimension vector such that $\alpha_j=\sum_{m=1}^s \alpha_{v_m}$, and let $X=(X_a)\in  \Rep(Q,\alpha)$. We define a $\alpha_j\times \alpha_j$-matrix $X^\prime$ by placing the $\alpha_j\times \alpha_{v_m}$ matrices $X_{a_m}$ side by side. We define $U^\prime$  to be the set
\begin{equation*}U^\prime=\left\{X=(X_a)\in  \Rep(Q,\alpha);\: \det X^\prime\neq 0 \right\}.\label{red2}\end{equation*}
Once again, we can use the $GL_{\alpha_j}(\cx)$-action to make $X^\prime$ equal to the identity matrix. It follows easily that we have the following isomorphism:
$$ U^\prime /GL_{\alpha_j}(\cx)\simeq  \Rep(Q^\prime,\alpha^\prime),$$
where $\alpha^\prime=(\alpha_1,\dots, \wh{\alpha_j},\dots,\alpha_k)$. This isomorphism is $G^\prime$-equivariant, where $G^\prime=\prod_{i\neq j} GL_{\alpha_i}$.

Moreover, once again, if  $\chi:G\to \cx^\ast$ is a character such that $$\chi|_{GL_{\alpha_j}(\cx)}=\prod_{m=1}^s\chi|_{GL_{\alpha_{v_m}}(\cx)},$$
then the $\chi$-semistable representations of $Q$ with invertible $X^\prime$ correspond to $\chi^\prime$-semistable representations of $Q^\prime$, where  $\chi^\prime:G^\prime \to \cx^\ast$  is the restriction of $\chi$ on $\prod_{i\neq j,v_1,\dots,v_s} GL_{\alpha_i}$, and $\chi^\prime=1$ on each $ GL_{\alpha_{v_m}}$. Thus, 
the corresponding GIT quotients are again isomorphic.
\par
Let $\Pi$ be a Poisson structure on $\cx Q$ and $\pi=\Psi(\Pi)$ the induced Poisson structure on $\Rep(Q,\alpha)$.Once again, we obtain a Poisson structure on $\Rep(Q^\prime,\alpha^\prime),$ since the latter is $ U^\prime /GL_{\alpha_j}(\cx)$. And again this Poisson structure is induced by a Poisson structure $\Pi^\prime$ on $\cx Q^\prime$. We can find $\Pi^\prime$ easily enough by following the procedure in the previous section. Let $S_j=\{b\in Q; h(b)=j\}$ and $T_j=\{b\in Q; t(b)=j\}$. When passing to $Q^\prime$, each arrow $b$ in $S_j$ or $T_j$ has been decomposed into either $s$ or $s^2$ (if $b$ is a loop) arrows $c_{b,i}$. We can decompose similarly the dual arrows $b^\ast$.
The condition for a covector to annihilate generators of the $GL_{\alpha_j}$-action is still \eqref{moment}. On the set where $X^\prime=1$, this corresponds to the following equation on arrows
$$ \sum_{m=1}^s a_m^\ast= \sum_{b\in T_j} \sum_{i,l} c_{b,i}^\ast c_{b,l}\:- \sum_{b\in S_j,\,b\neq a_m} \sum_{i,l} c_{b,i}  c_{b,l}^\ast.$$
Thus, as the $a_m$ begin at a different vertices,
\begin{equation} a_m^\ast= e_{v_m}\sum_{b\in T_j} \sum_{i,l} c_{b,i}^\ast c_{b,l}\:- e_{v_m}\sum_{b\in S_j,\,b\neq a_m} \sum_{i,l} c_{b,i}  c_{b,l}^\ast,
\label{oof}\end{equation}
$m=1,\dots,s$, where $e_{v_{m}}$ is the idempotent at vertex $v_m$. The element $\Pi^\prime\in \sV^2 Q^\prime$, which induces the Poisson structure on each $ U^\prime /GL_{\alpha_j}(\cx)$ is, therefore, obtained from $\Pi$ by substituting $e_{v_m}$ for each $a_m$ and \eqref{oof} for each $a_m^\ast$.
\par
The same argument as in the proof of Proposition \ref{lll} shows:
\begin{proposition}
If $\Pi$ is a Poisson structure on $\cx Q$, then $\Pi^\prime $ is Poisson structure on $\cx Q^\prime$.\hfill $\Box$ \label{llr}
\end{proposition}
\begin{remark} We can generalise further by introducing coefficients in the matrix $X^\prime$ (i.e. each block $X_{a_m}$ is multiplied by some $\epsilon_m\in \cx$).
\end{remark}


\section{Symplectic leaves for the Kronecker quiver with $3$ quivers \label{leaves}}

Generalised Kronecker quivers with $2$ vertices and $n$ arrows between those are of particular interest, given that they are natural objects in the construction of moduli spaces of coherent sheaves on projective schemes, as discovered by \'Alvarez-C\'onsul and King \cite{ACK}. The Kronecker quiver $Q_3$ with $3$ arrows is related even closer to moduli spaces of sheaves: owing to the results of Beauville \cite{Beau}, the moduli space of $\chi$-semistable representations of $Q_3$, with dimension vector $(k,k)$, and $\chi:GL_k(\cx)\times GL_k(\cx)\to \cx^\ast$, given by $\chi(g_1,g_2)=\det g_1g_2^{-1}$,  is isomorphic to the moduli space of acyclic $1$-dimensional sheaves on $\oP^2$ with Hilbert polynomial $P(u)=ku$. The isomorphism is provided by the following resolution of a sheaf $\sF$:
\begin{equation*}
 0\to \sO(-2)^k\stackrel{M}{\rightarrow}\sO(-1)^k\to \sF\to 0,
\end{equation*}
where $M=A\zeta_1+B\zeta_2+C\zeta_3$ is a matrix-valued linear polynomial, identified with the representation $(A,B,C)$ of $Q_3$.
\par
As observed in Example \ref{gl3}, there is a (family of) natural noncommutative Poisson structures on $\cx Q_3$, arising from the solution \eqref{agu} to the AYB equation, found by Aguiar. We want to study the induced Poisson structure on $\Rep(Q_3,(k,k))/\hspace{-1mm}/_\chi G$ in greater detail.
\par
First of all, as observed in the previous section, the Poisson structures on the open subsets, where one of the $A,B,C$ is invertible, are induced from noncommutative Poisson structures (described in Example \ref{exe}) on the free algebra $\cx\langle x,y\rangle$ on two letters. It is therefore enough to study these Poisson brackets. We shall concentrate on the second one, i.e. on
\begin{equation}
\Pi=-[x^2x^\ast+xyy^\ast, xx^\ast]+[y^\ast, x^\ast x-yy^\ast].\label{Pi_y}
\end{equation}
The corresponding open subset of $\Rep(Q_3,(k,k))$ is defined by $B$ being invertible, and
the induced Poisson bracket $\Psi(\Pi)$ on $\Mat_{k\times k}(\cx) \oplus \Mat_{k\times k}(\cx)$ is obtained from the one on $\Rep(Q_3,(k,k))$ by setting $X=AB^{-1}$, $Y=CB^{-1}$.
The corresponding open subset of $\Rep(Q_3,(k,k))/\hspace{-1mm}/_\chi G$ consists of sheaves $\sF$ such that $[0,1,0]\not\in \supp \sF$. 
\par
 We have already observed that the two homogenous terms in \eqref{Pi_y} define compatible Poisson brackets on $\cx\langle x,y\rangle$. Let us write
$$ \Pi_0=[xx^\ast,x]+[yx^\ast,y^\ast],\quad \Pi_{\infty}=[yy^\ast,yxx^\ast+y^2y^\ast].$$
We shall now consider the $1$-dimensional family of Poisson brackets on $\cx\langle x,y\rangle$, given by
$$\Pi_\epsilon=\Pi_0+\epsilon \Pi_\infty.$$
It is easy to see that a generic noncommutative Poisson bracket on $\cx\langle x,y\rangle$ induces a Poisson structure on $\Mat_{k\times k}(\cx) \oplus \Mat_{k\times k}(\cx)$, which has a Zariski open symplectic leaf. For $\Pi_0$ and $\Pi_\infty$ we can describe these leaves and the corresponding symplectic forms as follows:
\begin{lemma} The noncommutative Poisson brackets $\Pi_0$ and $\Pi_\infty$ on $\cx\langle x,y\rangle$, described in Corollary \ref{reduced}, induce generically nondegenerate Poisson structures on $\Mat_{k\times k}(\cx) \oplus \Mat_{k\times k}(\cx)$. The corresponding symplectic forms  are:
\begin{itemize}
\item[(i)] $\tr dY\wedge d(XY^{-1})\,$ for $\: [xx^\ast, x^\ast]+[yx^\ast, y^\ast]$;
\item[(ii)] $\tr d(YX)^{-1}\wedge dX\,$ for $\:[yy^\ast, yxx^\ast+y^2y^\ast]$.
\end{itemize}\label{four}
\end{lemma}
\begin{proof}  We begin with (i). The Poisson structure on $V=\Mat_{k\times k}(\cx) \oplus \Mat_{k\times k}(\cx)$ is given by
$$ T^\ast_{(X,Y)} V\ni (A,B)\longmapsto \left(XA-AX-BY,\, YA\right)\in  T_{(X,Y)} V.$$
It is clearly nondegenerate if $Y$ is invertible. We can then write $A=Y^{-1}dY$, and $B=XY^{-1}dYY^{-1}-Y^{-1}dYXY^{-1}-dXY^{-1}$. The symplectic form is now $\tr(AdX+BdY)$, i.e.
\begin{multline*}
 \tr\left(Y^{-1}dYdX+ XY^{-1}dYY^{-1}dY-Y^{-1}dYXY^{-1}dY-dXY^{-1}dY\right)=\\
\tr\left(Y^{-1}dY\wedge dX-Y^{-1}dY\wedge XY^{-1}dY\right)=\tr\left(dY\wedge (dX-XY^{-1}dY)Y^{-1}\right)=\\
\hspace*{3mm}  \tr dY\wedge d(XY^{-1}).\hfill
\end{multline*}
Similarly, for (ii), the Poisson structure viewed as map $T^\ast V\to TV$ is is given by:
$$ (A,B)\mapsto \left(YBYX,\, -YXAY+YBY^2-Y^2BY\right).$$
This is invertible if both $X$ and $Y$ are, and then:
$$ B=Y^{-1}dX X^{-1}Y^{-1},\enskip A=-X^{-1}Y^{-1}dYY^{-1}+X^{-1}Y^{-1}dXX^{-1}-X^{-1}dXX^{-1}Y^{-1}.$$
Hence the symplectic form $\tr(AdX+BdY)$ is
\begin{multline*}
 \tr\left( -X^{-1}Y^{-1}dYY^{-1}dX+X^{-1}Y^{-1}dXX^{-1}dX-X^{-1}dXX^{-1}Y^{-1}dX\right. +\\ \left. Y^{-1}dX X^{-1}Y^{-1}dY\right)=
\tr\left(-X^{-1}d(Y^{-1})\wedge dX-d(X^{-1})Y^{-1}\wedge dX\right)=\\\hspace*{3mm} \tr d(X^{-1}Y^{-1})\wedge dX.\hfill
\end{multline*}
\end{proof}
\begin{remark} The symplectic forms corresponding to the other Poisson structures in Corollary \ref{reduced} can be computed  similarly. For example, 
 $[xx^\ast, x^\ast]+[yy^\ast, x^\ast]$ gives $\tr d(Y^{-1}X)\wedge dY\,$, and $\:[yy^\ast,xy^\ast]-[y^\ast,x^2x^\ast+xy y^\ast]$ gives $\tr d(X^{-1})\wedge d(X^{-1}YX)\,$.
\end{remark}
\begin{remark}
 We would like to argue that the Poisson structure, induced on $\Mat_{k\times k}(\cx) \oplus \Mat_{k\times k}(\cx)$ by $[xx^\ast, x^\ast]+[yx^\ast, y^\ast]$, is related to the standard Poisson structure, given by $\tr dX\wedge dY$, in a way, which is similar to the relation between the trigonometric and rational Calogero-Moser systems. In particular, the functions $H_i=\tr X^i$, $i=1,\dots,k$, Poisson commute. Moreover, the symplectic leaf $\mu^{-1}(\sO_1)/GL_k(\cx)$, where $\mu(X,Y)=[Y,XY^{-1}]$ is the moment map for the symplectic form (i) in the above Lemma, and  $\sO_1$ consists of traceless matrices $A$ such that $1+A$ has rank $1$, should be viewed as a {\em trigonometric Calogero-Moser space}, analogous to the rational Calogero-Moser space  in \cite{Wilson}. Indeed, a simple computation shows that, if  $Y$ is diagonalisable, then $\tr X^2$ is the trigonometric Calogero-Moser Hamiltonian.
\end{remark}

\medskip

We now want to describe the  symplectic leaves for $\Pi_\epsilon$, in particular, its open dense symplectic leaf. In principle, one could proceed as in the proof of the last lemma, but the following approach is easier.
\begin{lemma} Let $\gamma= yxyy^\ast$ be a noncommutative vector field on  $\cx\langle x,y\rangle$. We have:
 $$[\gamma,\Pi_0]_{_{\sV}}=\Pi_\infty,\quad [\gamma,\Pi_\infty]_{_{\sV}}=0.$$
\label{deform}\end{lemma}
\begin{proof}
 Direct computation.
\end{proof}

From this, we obtain immediately:
\begin{corollary}
 Let $\Gamma=\Psi(\gamma)$ be the vector field induced by $\gamma$ on $V=\Mat_{k\times k}(\cx) \oplus \Mat_{k\times k}(\cx)$, and $\phi_t$ the corresponding family of local diffeomorphisms of $V$. Then $\Psi(\Pi_\epsilon)=(\phi_\epsilon)_\ast(\Psi(\Pi_0))$.\hfill $\Box$
\end{corollary}
A simple computation shows that
\begin{equation}
 \phi_t(X,Y)=\left(X, (Y^{-1}-tX)^{-1}\right).\label{dphi}
\end{equation}
In particular, we obtain, from Lemma \ref{four}:
\begin{corollary}
 The Poisson bracket $\Psi(\Pi_0+\epsilon\Pi_\infty)$ on $\Mat_{k\times k}(\cx) \oplus \Mat_{k\times k}(\cx)$ is generically nondegenerate. The corresponding symplectic form is:
\begin{equation}
 \tr  d(Y^{-1}-\epsilon X)^{-1}\wedge d\bigl(X(Y^{-1}-\epsilon X)\bigr). \qquad\qquad\Box\label{BB}
\end{equation}
\end{corollary}
We recall, from the beginning of the section, that $\Psi(\Pi_0+\Pi_\infty)$ corresponds to Poisson bracket on the open subset of the moduli space of acyclic $1$-dimensional sheaves on $\oP^2$ with Hilbert polynomial $P(u)=ku$. It follows, that we can study this Poisson structure via the   Poisson structure induced by $\Pi_0$ on $\Mat_{k\times k}(\cx) \oplus \Mat_{k\times k}(\cx)$ and the diffeomorphism \eqref{dphi} with $t=1$. In particular, generic symplectic leaves are of the form
$$ \left\{(X,Y);\: \bigl[(Y^{-1}-X)^{-1}, X(Y^{-1}-X)\bigr]\in O\right\}/GL_k(\cx),$$
for some adjoint orbit $O$ of $GL_k(\cx)$ (with the symplectic form  \eqref{BB} with $\epsilon=1$).
\begin{remark}
 Moduli spaces of semistable sheaves on $\oP^2$ carry other natural Poisson structures, called Mukai-Tuyrin-Bottacin Poisson structures, induced by Poisson brackets on $\oP^2$, i.e. by  sections of $K^\ast_{\oP^2}\simeq \sO(3)$ \cite{Mu, Tyu,Bo,HM2}. It is not clear to us whether the Poisson structure described above, and arising from a noncommutative quadratic Poisson bracket, is a Mukai-Tuyrin-Bottacin structure.
\end{remark}

\end{document}